\documentclass[a4paper,12pt,reqno]{amsart}       
\usepackage{amssymb,amsthm,amsmath}    
\usepackage[english]{babel}   
\usepackage{fullpage}    
\usepackage{hyperref}  

\title[Noncommutative products]
{Noncommutative products of Euclidean spaces} 
\date{August 2017; v2 November 2017}

\author{Michel Dubois-Violette, Giovanni Landi}
\address[]{\textit{Michel Dubois-Violette} 
\newline \indent  
Laboratoire de Physique Th\'eorique, UMR 8627,
\newline \indent
Universit\'e Paris XI,
B\^atiment 210, F-91 405 Orsay Cedex}
\email{Michel.Dubois-Violette@u-psud.fr}
\address[]{\textit{Giovanni Landi} 
\newline \indent  
Matematica,
\newline \indent 
Universit\`a di Trieste, 
Via A. Valerio, 12/1, 34127  Trieste, Italy 
\newline \indent
and INFN, Trieste, Italy}
\email{landi@units.it}

\numberwithin{equation}{section}
\newtheorem{theo}{Theorem}[section]
\newtheorem{lemm}[theo]{Lemma}
\newtheorem{prop}[theo]{Proposition}
\newtheorem{defi}[theo]{Definition}
\newtheorem{rema}[theo]{Remark}
\newtheorem{coro}[theo]{Corollary}

\newcommand{\ii}{{\mathrm{i} }}
\newcommand{\dd}{\mathrm{d}}
\newcommand{\nn}{\nonumber}

\newcommand{\ca}{\mathcal{A}}
\newcommand{\cc}{\mathcal{C}}
\newcommand{\cf}{\mathcal{F}}

\newcommand{\ccR}{\mathcal{R}}
\newcommand{\fraca}{{\mathfrak A}}
\newcommand{\fracg}{{\mathfrak g}} 
\newcommand{\IC}{{\mathbb C}}   
\newcommand{\IH}{{\mathbb H}}  
\newcommand{\II}{\mbox{\rm 1\hspace {-.6em} l}}
\newcommand{\IN}{{\mathbb N}}   
\newcommand{\IR}{{\mathbb R}}   
   
\newcommand{\IS}{{\mathbb S}}   
\newcommand{\IT}{{\mathbb T}}   
\newcommand{\IZ}{{\mathbb Z}}   

\newcommand{\Rt}{{\IR^4_\theta}} 
\newcommand{\Ct}{{\IC^2_\theta}} 
\newcommand{\car}{{\ca_R}} 
\newcommand{\cau}{{\ca^\pm_{{\bf u}}}} 
\newcommand{\capo}{{\ca^+_{{\bf u}}}} 
\newcommand{\cane}{{\ca^-_{{\bf u}}}} 

\DeclareMathOperator{\Hom}{Hom}
\DeclareMathOperator{\Ext}{Ext}

\DeclareMathOperator{\SO}{SO}        
\DeclareMathOperator{\SU}{SU}          
\DeclareMathOperator{\su}{\mathfrak{su}}          
\DeclareMathOperator{\U}{U}          
\DeclareMathOperator{\tr}{tr}

\newcommand{\beqa}{\begin{align}}
\newcommand{\eeqa}{\end{align}}
\newcommand{\beq}{\begin{equation}}
\newcommand{\eeq}{\end{equation}}

\begin{document}

\vskip -1cm
\begin{abstract}
We present natural families of coordinate algebras on noncommutative products of Euclidean spaces 
$\IR^{N_1} \times_{\ccR} \IR^{N_2}$. These coordinate algebras are quadratic ones associated with an
$\ccR$-matrix which is involutive and satisfies the Yang-Baxter equations. 
As a consequence they enjoy a list of nice properties, being regular of finite 
global dimension. Notably, we have eight-dimensional noncommutative euclidean spaces  
$\IR^{4} \times_{\ccR} \IR^{4}$. 
Among these, particularly well behaved ones have deformation parameter  
${\bf u} \in \IS^2$. Quotients include seven-spheres $\IS^{7}_{\bf u}$ as well as noncommutative quaternionic tori 
$\IT^{\IH}_{\bf u} = \IS^3 \times_{\bf u} \IS^3$. 
There is invariance for an action of $\SU(2) \times \SU(2)$ on the torus $\IT^{\IH}_{\bf u}$ in parallel with the action of $\U(1) \times \U(1)$ on a `complex' noncommutative torus $\IT^2_\theta$ which allows one to construct quaternionic toric noncommutative manifolds.  Additional classes of solutions are disjoint from the classical case.
\end{abstract}

\maketitle

\tableofcontents
\parskip = 1 ex

\thispagestyle{empty}

\section{Introduction and motivations}

Noncommutative tori, and in particular the two-torus $\IT^2_\theta$, are the subject of much study that 
continue to produce interesting results. Among other things, they are at the foundation of the  noncommutative toric manifolds of \cite{cl01} and \cite{cdv02}. In the same spirit  of these manifolds, at 
the centre of the present work there is a \emph{noncommutative quaternionic torus} and its properties and 
its role for related \emph{quaternionic toric} noncommutative manifolds. 

It is well known that one can identify the unit quaternions $\IH_1$
with the euclidean three-sphere $\IS^3$. Now, if $\IC_1$ denotes the unit complex number, 
the two-dimensional torus is written as $\IT^2 = \IC_1 \times \IC_1 = \IS^1 \times \IS^1$. 
In analogy with this, and for lack of a better name, we shall call a \emph{quaternionic tori} the cartesian product 
of two copies of $\IH_1$, that is 
$$
\IT^{\IH} = \IH_1 \times \IH_1 \simeq \IS^3 \times \IS^3. 
$$ 

And in analogy with the passage from the torus $\IT^2$ to the noncommutative one $\IT^2_\theta$ 
we shall pass from the torus 
$\IT^{\IH}$ to a noncommutative quaternionic torus 
$$
\IT^{\IH}_{\bf u} = \IH_1 \times_{\bf u} \IH_1 = \IS^3 \times_{\bf u} \IS^3 
$$ 
for a deformation parameter ${\bf u} \in \IS^2$. 
This is defined by duality as having coordinate 
$$
\ca(\IT^{\IH}_{\bf u}) = \ca(\IR^8_{\bf u}) /< \|x_1\|^2 - 1, \|x_2\|^2 -1> , 
$$
with $\ca(\IR^8_{\bf u})$ the coordinate algebra on a eight-dimensional noncommutative euclidean space  
$\IR^8_{\bf u}$ (generated by two sets of hermitian elements $(x_1^\mu, x_2^\alpha)$ with suitable commutation relations) and central quadratic elements $\|x_1\|^2 = \sum_\mu (x_1^\mu)^2$ and $\|x_2\|^2 
= \sum_\alpha (x_2^\alpha)^2$.  

In fact, the noncommutative euclidean space $\IR^8_{\bf u}$ is of the form $\IR^4 \times_{\bf u}\IR^4$, that  is to say it is a noncommutative product of commutative Euclidean spaces. This product space 
is a particular case of a more general construction of noncommutative products of Euclidean spaces $\IR^{N_1} \times_{\ccR} \IR^{N_2}$, for any two natural numbers $N_1$ and $N_2$.  Any such a noncommutative space is defined by duality as having 
coordinate algebra a quadratic algebra $\car$ whose generators $\{x^a\}$ are subject to commutation relations  given by an $\ccR$-matrix, $\ccR=(\ccR^{r \, s}_{a \, b})$, which is involutive and satisfies the Yang-Baxter equations. 
As a consequence the algebra enjoys a list of nice properties, being regular of finite  
global dimension $N_1+N_2$ and Gorenstein. Moreover, it has polynomial growth and the Poincar\'e series of $\car$ and of its Koszul dual algebra $\car^!$ are the classical ones; the algebra $\car$ 
is in fact Calabi--Yau, having a `volume form', that is a nontrivial top dimensional Hochschild cycle. An associated Clifford algebra $\cc\ell(\car)$ is a Koszul nonhomogeneous quadratic algebra with $\car^!$ as homogeneous part, of dimension 
$\dim \cc\ell(\car)=\dim \car^! =2^{N_1+N_2}$. 

\subsubsection{Notation}
We use Einstein convention of summing over repeated up-down indices.
Letters $\lambda, \mu, \nu, \rho, \dots$, from the middle of the Greek alphabet run in $\{1,\dots, N_1\}$ while letters $\alpha, \beta,\gamma, \delta, \dots$, from the beginning of the Greek alphabet run 
in $\{1,\dots, N_2\}$, with $N_1$ and $N_2$ arbitrary natural numbers. By an algebra we always mean an associative algebra and by a graded algebra we mean a $\IN$-graded algebra


\section{A class of regular quadratic algebras}

\subsection{General definitions}
We shall look for a $*$-algebra $\car$ generated by two sets of hermitian elements 
$x=(x_1, x_2)=(x_1^\lambda, x_2^\alpha)$, with $\lambda\in\{1,\dots,N_1\}$ and 
$\alpha\in \{1,\dots,N_2\}$, subject to relations:
\begin{align}\label{alde}
& x_1^\lambda x_1^\mu = x_1^\mu x_1^\lambda, 
\qquad x_2^\alpha x_2^\beta = x_2^\beta x_2^\alpha , \nn \\
& x_1^\lambda x_2^\alpha = R^{\lambda\alpha}_{\beta\mu} \, x_2^\beta x_1^\mu 
\end{align}
for a suitable `matrix' $R^{\lambda \mu}_{\nu \rho}$. 
In view of the hermiticity of the generators $x^\lambda_1, x^\alpha_2$, by taking the conjugate of the relations \eqref{alde},
we are lead to impose that 
\begin{equation}\label{real}
\overline{R}^{\lambda\alpha}_{\beta\mu} R^{\mu\beta}_{\gamma\nu}= \delta^\lambda_\nu \delta^\alpha_\gamma
\end{equation}
with \, $\bar{}$ \, the complex conjugation. 

For the equality \eqref{real} to be valid one needs to assume that the binomials $x_2^\alpha x_1^\lambda$ for any $\lambda, \alpha$, can be taken to be the independent ones (or alternatively one could take the binomials $
x_1^\lambda x_2^\alpha$). Furthermore, since any two components 
$x_1^\lambda$ and $x_1^\mu$ commute, the binomials $x_1^\lambda x_1^\mu$ for $\lambda \geq \mu$ (say) are independent and the same is valid for the binomials $x_2^\alpha x_2^\beta$ now for $\alpha \geq \beta$. 
Alternatively, for these commuting elements one could symmetrise in the indices: 
so the elements $x_1^\lambda x_1^\mu + x_1^\mu x_1^\lambda$ and $x_2^\alpha x_2^\beta + x_2^\beta x_2^\alpha$ are independent. 

Thus $\car$ is a graded quadratic algebra $\car=\oplus_n \car^n$ which is connected, that is 
$\car^0 = \IC\II$; the elements $x^\lambda_1,x^\alpha_2$ form a basis of the component $\car^1$; and  the elements $x^\lambda_1x^\mu_1$ with $\lambda\leq \mu$, $x^\alpha_2 x^\beta_2$ with $\alpha\leq \beta$ and $x^\lambda_1 x^\alpha_2$ (or $x^\alpha_2 x^\lambda_1$) form a basis of $\car^2$.
 
Given any such a quadratic algebra, its Koszul dual $\car^!$ is the quadratic algebra generated by dual 
elements $\theta=(\theta^1, \theta^2)=(\theta^1_\lambda, \theta^2_\alpha)$, with
$\lambda\in\{1,\dots,N_1\}$ and $\alpha\in \{1,\dots,N_2\}$, that is 
$\theta^1_\lambda(x_1^\mu) = \delta^\mu_\lambda$, $\theta^1_\lambda(x_2^\alpha) = 0$ and 
$\theta^2_\alpha(x_1^\lambda) = 0$, $\theta^2_\alpha(x_2^\beta) = \delta_\alpha^\beta$. 
These generators are subject to natural relations which are `orthogonal' to be ones in \eqref{alde}:
\begin{align}\label{aldek}
& \theta^1_\lambda \theta^1_\mu = -\theta^1_\mu \theta^1_\lambda, 
\qquad\theta^2_\alpha \theta^2_\beta = -\theta^2_\beta \theta^2_\alpha , \nn \\
& \theta^2_\beta \theta^1_\mu = -R^{\lambda \alpha}_{\beta \mu} \, \theta^1_\lambda \theta^2_\alpha .
\end{align}

\subsection{An involutive $\ccR$ matrix} 
Let us put together the coordinates by defining 
$x^a$ for $a\in \{1,2, \dots, N_1+N_2\}$ with $x^\lambda=x^\lambda_1$ and 
$x^{\alpha+N_1}=x^\alpha_2$. Then, the commutation relations \eqref{alde} together with the conditions $x^\alpha_2x^\lambda_1=\overline{R}^{\lambda\alpha}_{\beta\mu} x^\mu_1x^\alpha_2$ can be written in the form
 \beq\label{aldeb}
x^a x^b = \ccR^{a \, b}_{c \, d} \,\, x^c x^d . 
\eeq
In view of \eqref{alde} and \eqref{real} we get the following expression for $\ccR^{a \, b}_{c \, d}$, 
\begin{align}\label{rtot}
 \ccR^{\lambda \mu}_{\tau \rho} & = \delta^\lambda_\rho \, \delta^\mu_\tau ,
\qquad  \ccR^{\gamma \delta}_{\alpha \beta} = \delta^\gamma_\beta \, \delta^\delta_\alpha \nn \\ 
\ccR^{\lambda \alpha}_{\beta \mu} & = R^{\lambda \alpha}_{\beta \mu} ,\qquad 
\ccR^{\alpha \lambda}_{\mu \beta} = \overline{R}^{\lambda \alpha}_{\beta \mu} \nn \\
\ccR^{\lambda \mu}_{\alpha \nu} & = \ccR^{\lambda \mu}_{\alpha \beta} = \ccR^{\lambda \mu}_{\nu \beta} =0 , \nn \\
\ccR^{\alpha \gamma}_{\lambda \beta} & = \ccR^{\alpha \gamma}_{\lambda \mu} = \ccR^{\alpha \gamma}_{\beta \mu} = 0 , \nn \\
\ccR^{\lambda \alpha}_{\mu \nu} & =  \ccR^{\lambda \alpha}_{\beta \gamma} 
= \ccR^{\lambda \alpha}_{\mu \beta} =0 , \nn \\
\ccR^{\alpha \lambda}_{\mu \nu} & = \ccR^{\alpha \lambda}_{\beta \gamma} = \ccR^{\alpha \lambda}_{\beta \mu} = 0 .
\end{align}

\begin{lemm}\label{involr}
The $\ccR$ matrix given as in \eqref{rtot} is involutive, that is
$$
\ccR^2 = \II_{N_1} \otimes \II_{N_2}. 
$$
\end{lemm}
\begin{proof} 
This is just a direct computation using the fact that  
$(R^{-1})^{\alpha \lambda}_{ \mu \beta} = \overline{R}^{\lambda \alpha}_{\beta \mu}$ from \eqref{real}. 
\end{proof}

\subsection{The Yang-Baxter condition} 
We next impose that the matrix $\ccR$ satisfies the Yang-Baxter equation, that is it satisfies the equation
\beq\label{rvan}
(\ccR \otimes \II) (\II \otimes \ccR) (\ccR \otimes \II) = (\II \otimes \ccR)(\ccR \otimes \II) (\II \otimes \ccR) .
\eeq
In component this can be written as 
\beq\label{rvan1}
\big( (\ccR \otimes \II) (\II \otimes \ccR) (\ccR \otimes \II) - (\II \otimes \ccR)(\ccR \otimes \II) (\II \otimes \ccR) \big)^{a b c} = 0
\eeq
for indices $a, b, c \in\{1,2, \dots, N_1+N_2\}$.

\begin{prop}
The Yang-Baxter equation \eqref{rvan1} for $\ccR$ is equivalent to the following 
\begin{equation}\label{yb1}
\left\{
\begin{array}{ll}
R^{\lambda\alpha}_{\gamma\rho} R^{\rho\beta}_{\delta\mu} 
= R^{\lambda\beta}_{\delta\rho} R^{\rho\alpha}_{\gamma\mu} & \qquad \mbox{for indices} \qquad (a, b, c)=(\lambda\alpha\beta)\\
\\
\overline{R}^{\lambda\alpha}_{\gamma\rho} \overline{R}^{\rho\beta}_{\delta\mu} 
= \overline{R}^{\lambda\beta}_{\delta\rho}\overline{R}^{\rho\alpha}_{\gamma\mu} & \qquad \mbox{for indices} \qquad (a, b, c)= (\alpha\beta\lambda)\\
\\
\overline{R}^{\lambda\alpha}_{\gamma\rho} R^{\rho\beta}_{\delta\mu} 
= R^{\lambda\beta}_{\delta\rho} \overline{R}^{\rho\alpha}_{\gamma\mu} & \qquad \mbox{for indices} \qquad (a, b, c)= (\alpha\lambda\beta)
\end{array}
\right.
\end{equation}
and
\begin{equation} \label{yb2}
\left\{
\begin{array}{ll}
R^{\lambda\alpha}_{\gamma\nu} R^{\mu\gamma}_{\beta\rho} 
= R^{\mu\alpha}_{\gamma\rho} R^{\lambda\gamma}_{\beta\nu} & \qquad \mbox{for indices} \qquad (a, b, c)= (\lambda\mu\alpha)\\
\\
\overline{R}^{\lambda\alpha}_{\gamma\nu} \overline{R}^{\mu\gamma}_{\beta\rho} 
= \overline{R}^{\mu\alpha}_{\gamma\rho}\overline{R}^{\lambda\gamma}_{\beta\nu} & \qquad \mbox{for indices} \qquad (a, b, c)=(\alpha\lambda\mu)\\
\\
R^{\lambda\alpha}_{\gamma\nu} \overline{R}^{\mu\gamma}_{\beta\rho} 
= \overline{R}^{\mu\alpha}_{\gamma\rho}  R^{\lambda\gamma}_{\beta\nu} & \qquad \mbox{for indices} \qquad (a, b, c)= (\lambda\alpha\mu)
\end{array}
\right.
\end{equation}

\noindent
for the matrices $R^{\lambda\alpha}_{\beta\mu}$ and $\overline{R}^{\lambda\alpha}_{\beta\mu}$.
\end{prop}
\begin{proof} 
This is just a direct computation by using the expressions \eqref{rtot}. 
\end{proof}

\subsection{Regularity of the algebras} 
From now we assume that the matrix $R$ of relations \eqref{alde} for the algebra $\car$ satisfies conditions \eqref{real}, \eqref{yb1} and \eqref{yb2}, and study the regularities properties of $\car$. 
We refer to Appendix~\ref{quad} for the relevant notions used in this section. 
One has the following result concerning the structure of the algebra $\car$.

\begin{theo}
The algebra $\car$ is Koszul of global dimension $N_1+N_2$ and Gorenstein.
Moreover, it has polynomial growth and the Poincar\'e series of $\car$ and of $\car^!$ are classical, that is one has 
$$
P_{\car}(t) = \frac{1}{(1-t)^{N_1+N_2}}  \qquad \mbox{and} \qquad P_{\car^!}(t) = (1+t)^{N_1+N_2} \, .
$$
\end{theo}
\begin{proof}
This theorem follows from the more general result of \cite{gu90} concerning the Hecke symmetries since, due to Lemma~\ref{involr} we are in the involutive case, that is the case of a permutation symmetry. Let us explain shortly the method of \cite{gu90} in our case. 

There are two quadratic algebras naturally associated to the matrix  $\ccR$  (satisfying $\ccR^2 = \II \otimes \II$). 
On the one hand, the generalization of the symmetric algebra $S_{\ccR}$ generated by the elements $x^a$ with relations $x^a x^b = \ccR^{a \, b}_{c \, d} \,\, x^c x^d$ which then coincides with $\car$. On the other hand, the generalization of the exterior algebra $\Lambda_{\ccR}$ generated by the elements $(\dd x)^a$ with relations 
$(\dd x)^a (\dd x)^b = - \ccR^{a \, b}_{c \, d} \,\, (\dd x)^c (\dd x)^d$ which then coincides with the dual 
$(\car^!)^\ast$
of the Koszul dual $\car^!$ of $\car$ as graded vector space. 

\noindent
On the free graded left $S_{\ccR}=\car$ module
$$
S_{\ccR} \otimes \Lambda^{\bullet}_{\ccR} 
$$
one has two differentials, the Koszul differential 
$$
\delta :  S^k_{\ccR} \otimes \Lambda^l_{\ccR} \to S^{k+1}_{\ccR} \otimes \Lambda^{l-1}_{\ccR}
$$
and the generalization of the exterior differential 
$$
\dd :  S^k_{\ccR} \otimes \Lambda^l_{\ccR} \to S^{k-1}_{\ccR} \otimes \Lambda^{l+1}_{\ccR}
$$
which satisfy the classical relation
$$
( \delta  \dd + \dd \delta ) \alpha = (k+l) \alpha \qquad \mbox{for} \quad \alpha \in S^k_{\ccR} \otimes \Lambda^l_{\ccR}. 
$$

\noindent
This implies both the koszulity of $S_{\ccR}=\car$ and a generalization of the Poincar\'e lemma. 
By construction, the elements $(x^{a_1} \cdots x^{a_p})$ for $a_1 \leq a_1 \cdots \leq x^{a_p}$, for $p\in \IN$,  
form a homogeneous basis of the graded vector space $\car$ which implies that the latter 
has polynomial growth and more generally, that its Poincar\'e series is given by 
$$
P_{\car}(t) := \sum_{p\in\IN} \dim_{\IC}(\car(p)) \, t^p= \frac{1}{(1-t)^{N_1+N_2}}
$$ 
while the Poincar\'e series of $\car^!$ is $P_{\car^!}(t) := (P_{\car}(-t))^{-1} = (1+t)^{N_1+N_2}$, as stated. (On this hand, the elements $\theta_{a_1}\dots \theta_{a_p}$ for $a_1<\dots<a_p$ and $p\in \{1,\dots,N_1+N_2\}$ form a basis of $\car^!$.) As a consequence the ground field $\IC$ has projective dimension  $N_1+N_2$ which coincides with the global dimension
$$
\mbox{gl}(\car) = N_1+N_2 ,
$$
of $\car$ as for all graded algebras \cite{ca58}, and coincides also with its Hochschild dimension \cite{be05}. 

Next, by applying the functor $\Hom_{\car}(\, \cdot \, , \car)$ to the Koszul chain complex of free left 
$\car$-modules $(\car \otimes (\car^!)^*_\bullet, \delta)$, one obtains a cochain complex of right 
$\car$-modules
the cohomology of which is, by definition $\Ext^\bullet_\car(\IC,\car)$.  It follows from the results above that one has the Gorestein properties:
$$
\Ext^n_\car(\IC,\car) = \delta^n_{N_1+N_2} \, \IC .
$$
This finishes the proof. \end{proof}
\begin{rema}
\textup{
The algebra $\car$ is even a Calabi--Yau algebra as defined in \cite{gi06}. This is because any generator 
of the top one-dimensional space $((\car^!)_{N_1 + N_2})^\ast$ is a graded cyclic potential 
(a graded cyclic pre-regular multilinear form in the sense of \cite{dv07}) for the algebra 
$\car$ (see \cite{dvl17}).
}
\end{rema}

\begin{rema}
\textup{
The methods of \cite{gu90} were also used in \cite{wa93} and we refer to \cite{be00} for an alternative approach connected with the notion of confluence. 
}
\end{rema}

\section{Noncommutative products of Euclidean spaces}

\subsection{Noncommutative product of spaces $\mathbb R^{N_1}$ and $\mathbb R^{N_2}$}
The classical (commutative) solution $R=R_0$ of conditions \eqref{real}, \eqref{yb1} and \eqref{yb2}
is given by 
$$
(R_0)^{\lambda\alpha}_{\beta\mu}=\delta^\lambda_\mu  \delta^\alpha_\beta
$$ 
and the corresponding algebra $\ca_{R_0}$ is the algebra of polynomial functions on the product 
 $\IR^{N_1}\times \IR^{N_2}$. For this reason, we define the \emph{noncommutative product} 
 of $\mathbb R^{N_1}$ and $\mathbb R^{N_2}$  
\beq\label{ncp}
\mathbb R^{N_1}\times_{R} \mathbb R^{N_2}
\eeq
to be the `space dual' to the (coordinate) algebra $\car$ for a general matrix $R$ 
(satisfying the conditions \eqref{real}, \eqref{yb1} and \eqref{yb2}).

\subsection{Noncommutative Euclidean spaces}
With notations as before: the generators are labelled as 
$x^a = (x^\lambda=x^\lambda_1, \,  x^{\alpha+N_1}=x^\alpha_2)$
for $a\in \{1,2, \dots, N_1+N_2\}$ and the relations as in \eqref{alde} (or \eqref{aldeb})
we have the following result. 
\begin{theo} 
The following conditions $\mathrm{(i)}$ $\mathrm{(ii)}$ and $\mathrm{(iii)}$ are equivalent:
\begin{align}\label{cerad}
\mathrm{(i)} & \qquad\qquad \sum^{N_1+N_2}_{a=1} (x^a)^2=\sum^{N_1}_{\lambda=1} (x^\lambda_1)^2 + \sum^{N_2}_{\alpha=1} (x^\alpha_2)^2 \qquad  \mbox{is central in} \quad \car,  \nn \\
\mathrm{(ii)} & \qquad\qquad \sum^{N_1}_{\lambda=1}(x^\lambda_1)^2 \quad \mbox{and} \quad \sum^{N_2}_{\alpha=1}(x^\alpha_2)^2
\qquad  \mbox{are in the center of} \quad  \car, \nn \\
\mathrm{(iii)} & \qquad\qquad \sum^{N_1}_{\lambda=1} R^{\lambda\gamma}_{\beta\nu} 
R^{\lambda\beta}_{\alpha \mu}= \delta^\gamma_\alpha \delta_{\mu\nu} \qquad \mbox{and} \qquad  \sum^{N_2}_{\alpha=1} R^{\lambda\alpha}_{\beta\rho} R^{\rho\alpha}_{\gamma\mu}=\delta^\lambda_\mu \delta_{\beta\gamma}. 
\end{align}
\begin{proof}
Since any two components $x_1^\lambda$ and $x_1^\mu$ commute among themselves, and the same is true for any two components $x_2^\alpha$ and $x_2^\beta$, the equivalence of points $\mathrm{(i)} $ and $\mathrm{(ii)} $ follows.
For the rest one just computes using the defining relations \eqref{alde}
$$
\sum_\lambda (x_1^\lambda)^2 x_2^\gamma = \sum_\lambda x_1^\lambda x_1^\lambda x_2^\gamma = 
\sum_\lambda x_1^\lambda R^{\lambda \gamma}_{\beta \nu} x_2^\beta x_1^\nu 
=\sum_\lambda R^{\lambda \gamma}_{\beta \nu} R^{\lambda \beta}_{\alpha \mu}  \,\, 
x_2^\alpha x_1^\mu x_1^\nu .
$$
Asking that the right-hand side be $\sum_\lambda x_2^\gamma (x_1^\lambda)^2$ yields the first condition in point 
$\mathrm{(iii)}$. The second one is obtained along similar lines. 
\end{proof}
\end{theo}

We take the matrix $R$ to satisfies \eqref{cerad} also and define the \emph{noncommutative product of the Euclidean spaces} $\IR^{N_1}$ and $\IR^{N_2}$ 
to be dual of the algebra $\car$ with these additional conditions. 
Clearly, the \eqref{cerad} are satisfied by the classical $R=R_0$. The algebra
$\car$ generalizes the algebra of polynomial functions on the product $\IR^{N_1} \times \IR^{N_2}$.
 
The centrality conditions in \eqref{cerad} taken together with the reality conditions \eqref{real} leads to the following conditions on the matrix $R^{\lambda\alpha}_{\beta\mu}$. 
\begin{prop}
By using \eqref{real} and \eqref{cerad} one obtains
\begin{equation}\label{eucl0}
R^{\lambda\beta}_{\alpha\mu}=R^{\mu\alpha}_{\beta\lambda}=(R^{-1})^{\beta\mu}_{\lambda\alpha}
=\overline{R}^{\mu\beta}_{\alpha\lambda} . 
\end{equation}
In turn this implies that relations \eqref{yb1} and \eqref{yb2} reduce to
\begin{equation}\label{eucl1}
R^{\lambda\alpha}_{\beta\rho} R^{\rho\delta}_{\gamma\mu} = R^{\lambda\delta}_{\gamma\rho}R^{\rho\alpha}_{\beta\mu} \qquad \mbox{and} \qquad 
R^{\lambda\alpha}_{\gamma\nu} R^{\mu\gamma}_{\beta\rho} = R^{\mu\alpha}_{\gamma\rho} R^{\lambda\gamma}_{\beta\nu}
\end{equation}
that is the first relation of \eqref{yb1} and the first relation of \eqref{yb2}.
\end{prop}
\begin{proof}
We know from \eqref{real} that $(R^{-1})^{\beta\mu}_{\lambda\alpha}
=\overline{R}^{\mu\beta}_{\alpha\lambda}$. When comparing the first sum in the point (iii) 
of \eqref{cerad} with \eqref{real}, we see there is a `transposition'in the indices $\lambda, \mu$
and this leads to $R^{\lambda\beta}_{\alpha\mu}=(R^{-1})^{\beta\mu}_{\lambda\alpha}$. Similarly,
a comparison of the second sum in the point (iii) of \eqref{cerad} with \eqref{real}, shows 
a `transposition' in the indices $\alpha, \beta$ and this gives
$R^{\mu\alpha}_{\beta\lambda}=(R^{-1})^{\beta\mu}_{\lambda\alpha}$. 
The fact that then the relations \eqref{yb1} and \eqref{yb2} reduce to the two in \eqref{eucl1} is evident. 
\end{proof}

\begin{coro}
\bigskip
The relations \eqref{alde} define (the algebra of a) noncommutative product of a $N_1$-dimensional with a 
$N_2$-dimensional Euclidean spaces if and only if the matrix $R^{\lambda\alpha}_{\beta\mu}$ satisfy 
relations \eqref{eucl0} and \eqref{eucl1}.

\end{coro}

 \subsection{Noncommutative spheres and products of spheres}
 With the quadratic elements $\| x_1 \|^2 = \sum^{N_1}_{\lambda=1}(x^\lambda_1)^2$ and 
$\| x_2 \|^2 = \sum^{N_2}_{\alpha=1}(x^\alpha_2)^2$ of $\car$ being central, one may consider the quotient algebra
$$
\car / (\{ \| x_1 \|^2-\II,  \| x_2 \|^2-\II\})
$$
which defines by duality the noncommutative product 
$$
\IS^{N_1-1} \times_R \IS^{N_2-1}
$$
of the classical spheres $\IS^{N_1-1}$ and $\IS^{N_2-1}$. Indeed, for $R=R_0$, the above quotient is the restriction to $\IS^{N_1-1}\times \IS^{N_2-1}$ of the polynomial functions on 
$\IR^{N_1+N_2}$. 

Furthermore, with the central quadratic element $\| x \|^2 = \sum^{N_1+N_2}_{a=1} (x^a)^2= \| x_1 \|^2 +  \| x_2 \|^2$,
one may also consider the quotient of $\car$
$$
\car/( \| x \|^2-\II).
$$
This defines (by duality) the noncommutative $(N_1+N_2-1)$-sphere 
$$
\IS^{N_1+N_2-1}_R
$$
(a `subspace' of the noncommutative product of $\IR^{N_1} \times_R \IR^{N_2}$) which will be shown \cite{dvl17} to be a noncommutative spherical manifold in the sense of \cite{cl01,cdv02}.

\section{Clifford algebras}
\begin{defi}
The (generalised) Clifford algebra $\cc\ell(\car)$ is the algebra generated by elements 
$\Gamma=(\Gamma^1_\mu$, $\Gamma^2_\alpha)$, 
for $\lambda\in\{1,\dots,N_1\}$ and 
$\alpha\in \{1,\dots,N_2\}$, subject to relations
\begin{align}\label{aldecl}
& \Gamma^1_\lambda \Gamma^1_\mu + \Gamma^1_\mu \Gamma^1_\lambda 
= 2 \delta_{\lambda \mu} \II , 
\qquad \Gamma^2_\alpha \Gamma^2_\beta + \Gamma^2_\beta \Gamma^2_\alpha 
= 2 \delta_{\alpha \beta} \II , \nn \\
& \Gamma^2_\beta \Gamma^1_\mu + R^{\lambda \alpha}_{\beta \mu} \, 
\Gamma^1_\lambda \Gamma^2_\alpha = 0.
\end{align}
\end{defi}
With the definitions \eqref{alde} and \eqref{aldecl}, the following two (equivalent) propositions do not depend on the explicit choice of the matrix $R^{\lambda \mu}_{\nu \rho}$.

\begin{prop}
In the algebra $\cc\ell(\car) \otimes \car$ one has that
\begin{align*}
(\Gamma^1_\lambda \otimes x_1^\lambda)^2 & = \II \otimes \sum_\lambda (x_1^\lambda)^2 = \II \otimes \|x_1\|^2,  \nn \\
(\Gamma^2_\alpha \otimes x_2^\alpha)^2 & = \II \otimes \sum_\alpha (x_2^\alpha)^2 = \II \otimes \|x_2\|^2 , \nn \\
(\Gamma^1_\lambda \otimes x_1^\lambda)(\Gamma^2_\alpha \otimes x_2^\alpha ) & + 
(\Gamma^2_\alpha \otimes x_2^\alpha )(\Gamma^1_\lambda \otimes x_1^\lambda) = 0 .
\end{align*}
\begin{proof}
Since the components $x_1^\lambda$ commute among themselves, one
computes
$$
(\Gamma^1_\lambda \otimes x_1^\lambda)^2 
= \Gamma^1_\lambda \Gamma^1_\mu \otimes x_1^\lambda x_1^\mu 
= \tfrac{1}{2} ( \Gamma^1_\lambda \Gamma^1_\mu + \Gamma^1_\mu \Gamma^1_\lambda) 
\otimes x_1^\lambda x_1^\mu = \II \otimes \delta_{\lambda \mu} x_1^\lambda x_1^\mu = \II \otimes \|x_1\|^2, 
$$
having used the first relation in \eqref{aldecl}. A similar computation works for the quadratic element 
$\|x_2\|^2$. As for the last identity, 
one computes:
\begin{align*}
\Gamma^1_\lambda \Gamma^2_\alpha \otimes x_1^\lambda x_2^\alpha 
+ \Gamma^2_\beta \Gamma^1_\mu \otimes x_2^\beta x_1^\mu
= \Gamma^1_\lambda \Gamma^2_\alpha \otimes R^{\lambda \alpha}_{\beta \mu} \, x_2^\beta x_1^\mu
-R^{\lambda \alpha}_{\beta \mu} \, \Gamma^1_\lambda \Gamma^2_\alpha 
\otimes x_2^\beta x_1^\mu = 0
\end{align*}
using now the last ones of the conditions \eqref{alde} and \eqref{aldecl}.  
\end{proof}
\end{prop}

\begin{prop}
In the algebra $\cc\ell(\car) \otimes \car$, one has that 
\beq\label{gsxs}
(\Gamma(x))^2 = \II \otimes \|x\|^2 .
\eeq
where 
$$
\Gamma(x) := \Gamma^1_a \otimes x^a = \Gamma^1_\lambda \otimes x_1^\lambda 
+ \Gamma^2_\alpha \otimes x_2^\alpha 
$$

\noindent
and $\|x\|^2 = \sum_a (x^a)^2 = \|x_1\|^2 + \|x_2\|^2$
the square norm of $x=(x_1,x_2)$. 
\begin{proof}
One just computes:
\begin{align*}
(\Gamma(x))^2 & = (\Gamma^1_\lambda \otimes x_1^\lambda)^2 + (\Gamma^2_\lambda \otimes 
x_2^\lambda)^2 + 
(\Gamma^1_\lambda \otimes x_1^\lambda)(\Gamma^2_\alpha \otimes x_2^\alpha )  + 
(\Gamma^2_\alpha \otimes x_2^\alpha )(\Gamma^1_\lambda \otimes x_1^\lambda) \\
& =  \II \otimes \|x_1\|^2  + \II \otimes \|x_2\|^2 = \II \otimes \|x\|^2 
\end{align*}
using the same computation as in the previous proposition. 
\end{proof}
\end{prop}

The algebra $\cc\ell(\car)$ is a nonhomogeneous quadratic algebra with $\car^!$ as homogeneous quadratic part (see Appendix~\ref{quad}). It is not $\IN$-graded but only $\IZ_2$-graded and filtered with
$$
\cf ^n=F^n(\cc\ell(\car))=\{ \mbox{elements of degree in}\  \Gamma\leq n\}
$$
One has a surjective canonical homomorphism of graded algebras
\beq \label{canc}
\mbox{can} : \car^!\rightarrow\ \mbox{gr}(\cc\ell(\car))=\oplus_{n\in\IN} \, \cf^n/\cf^{n-1}
\eeq
which induce the isomorphism of vector spaces
$$
(\car^!)^1\simeq \cf^1/\cf^0
$$
The algebra $\cc\ell(\car)$ has the following Poincar\'e--Birkhoff--Witt (PBW) property: 
\begin{prop}\label{PBW}
The homomorphism \eqref{canc} is an isomorphism of graded algebras.
\end{prop}
\begin{proof}
With respect to the results recalled in Appendix~\ref{quad}, now in \eqref{PR} one has $\psi_1=0$.
Since $\car^!$ is Koszul one needs to show that the two conditions (i) and (ii) in \eqref{i-ii} are satisfied. 
This is easily established being the element $\|x\|^2 = \sum_a (x^a)^2$ central in $\car$. 
\end{proof}

As a consequence of this isomorphism, the koszulity of $\car^!$ yields that
the nonhomogeneous quadratic algebra $\cc\ell(\car)$ is Koszul as well \cite{bg06} (see also \cite{dv14}). In turn, this implies that
$\dim \cc\ell(\car)=\dim \car^! =2^{N_1+N_2}$. 
Finally, if $\cc\ell(N)$ denotes the usual Clifford algebra of the euclidean space $\IR^N$, 
one has isomorphisms:
$$
\left\{
\begin{array}{l}
\cc\ell(N_1)\simeq\ \mbox{subalgebra of}\ \cc\ell(\car)\ \mbox{generated by the}\ \Gamma^1_\lambda \\ ~\\
\cc\ell(N_2)\simeq\ \mbox{subalgebra of}\ \cc\ell(\car)\ \mbox{generated by the}\ \Gamma^2_\alpha
\end{array}
\right.
$$
as well as
$$
\cc\ell(\car)\simeq \cc\ell(N_1+N_2) .
$$

As a consequence of having a basis from the previous proposition, one has the following.
\begin{coro}
The Clifford algebra $\cc\ell(\car)$ being defined with relations \eqref{aldecl} for its generators $\Gamma$'s, the relations \eqref{alde} for the algebra $\car$ are equivalent to \eqref{gsxs}, that is one has the relations \eqref{alde} for the $x$'s if and only if  $(\Gamma(x))^2 = \II \otimes \|x\|^2$. 
\end{coro}
 
\section{Families of noncommutative euclidean planes} 
Before we proceed we have a look at the four-dimensional case.

\subsection{The four dimensional case}\label{se:4d}
A one-parameter family of noncommutative two-dimensional complex spaces $\Ct$ was introduced in \cite{cl01}. With $\theta\in \IR$, the coordinate unital $*$-algebra of $\Ct$
is generated by two normal elements $z_1, z_2$, that is 
$$
[z_1, z^*_1]=0=[z_2, z^*_2], 
$$
with relations:
$$
z_1 z_2 = e^{\ii \theta} z_2 z_1 , \qquad z_1 z^*_2 = e^{- \ii \theta} z^*_2 z_1 ,
$$
together with their $*$-conjugates. One easily checks that both $z_1 z^*_1$ and $z_2 z^*_2$ are in the center of the algebra. Sending $\theta \to- \theta$ results into an isomorphic algebra at the expenses of exchanging 
$z_1 \leftrightarrow z_2$. Thus this family is really parametrised by 
$\IS^1 / \IZ_2 \simeq \IS^1$. 

One can pass to a noncommutative four-dimensional euclidean spaces 
$\Ct \simeq \Rt$ via hermitian generators $(x_1^1, x_1^2)$ and $(x_2^1, x_2^2)$
by writing
$$
z_1 = x_1^1 + \ii \, x_1^2 , \qquad z_2 = x_2^1 + \ii \, x_2^2 .
$$
Then the algebra relations are easily found to be given by 
\begin{align*}
& x_1^1 x_1^2 = x_1^2 x_1^1 , \qquad x_2^1 x_2^2 = x_2^2 x_2^1  \nn \\
& x_1^\lambda x_2^\alpha = R^{\lambda \alpha}_{\beta \mu} \, x_2^\beta x_1^\mu ,
\end{align*}
with matrix $R^{\lambda \mu}_{\nu \rho}$ given explicitly by
\beq\label{r4d}
R^{\lambda \alpha}_{\beta \rho} = \cos\theta \, \delta^\lambda_\rho \delta^\alpha_\beta + 
\ii \sin\theta \, (C)^\lambda_\rho (D)^\alpha_\beta \qquad \mbox{and} \quad 
C = \begin{pmatrix} 
0 & - 1 \\
1 & 0 
\end{pmatrix}, \quad 
D = \begin{pmatrix} 
0 & 1 \\
- 1 & 0 
\end{pmatrix} . 
\eeq
In fact, in the notation of \eqref{ncp} we are really defining $\Rt$ as a noncommutative product 
$\Rt=\IR^2 \times_\theta \IR^2$. As for the matrices in \eqref{r4d}, what we have is that 
$$
C \otimes D = \sigma_2 \otimes \sigma_2 \qquad \mbox{with} \quad \sigma_2 =
\begin{pmatrix} 
0 & -\ii \\
\ii & 0 
\end{pmatrix}  
$$
the antisymmetric Pauli matrix. As mentioned before, mapping $\theta \to -\theta$ yields an isomorphic algebra now 
under the exchange $x_1^1 \leftrightarrow x_2^1$ and $x_1^2 \leftrightarrow x_2^2$. 

\subsection{The ansatz $ABCD$}
We try to generalise \eqref{r4d} in higher dimensions by looking for a matrix $R^{\lambda \mu}_{\nu \rho}$
in the defining relations \eqref{alde} of the form
\beq\label{r4dgen}
R^{\lambda \alpha}_{\beta \mu} 
= A^{\lambda}_ {\mu} \, B^{\alpha}_{\beta} + \ii \, C^{\lambda}_ {\mu} \, D^{\alpha}_{\beta} 
\eeq
for $N_1 \times N_1$ matrices $A,C$ and $N_2 \times N_2$ matrices $B,D$.

\begin{prop} 
The matrix $R^{\lambda \alpha}_{\beta \mu}$ of the form in \eqref{r4dgen} satisfies the reality condition \eqref{real} if and only if the matrices $A,B,C,D$ are such that
\beq\label{real8}
\overline{A} A \otimes \overline{B}  B + \overline{C}  C \otimes \overline{D}  D +
\ii \, \big( \overline{A} C \otimes \overline{B}  D - \overline{C}  A \otimes \overline{D}  B \big) = \II_{N_1} \otimes \II_{N_2}.
\eeq
being  \, $\bar{}$ \, complex conjugation. 
\begin{proof}
A direct computation of \eqref{real} yields
\beq
(\overline{A} A)_\tau^\lambda \, (\overline{B}  B)_\gamma^\alpha + (\overline{C}  C)_\tau^\lambda \, 
(\overline{D}  D)_\gamma^\alpha + 
\ii \, \Big[ (\overline{A} C)_\tau^\lambda \, (\overline{B}  D)_\gamma^\alpha  - 
(\overline{C}  A)_\tau^\lambda \, (\overline{D}  B)_\gamma^\alpha \Big]
=  \delta_\tau^\lambda \delta_\gamma^\alpha
\eeq
which is just the condition \eqref{real8} in components. 
\end{proof}
\end{prop}

\begin{prop}\label{pro-comm}
The matrix $R^{\lambda \alpha}_{\beta \mu}$ of the form in \eqref{r4dgen} satisfies the conditions in \eqref{eucl1} if and only if the matrices $A,B,C,D$ are such that
$$
[A, C] = 0 = [B,D] .
$$
\begin{proof}
Starting with the first condition in \eqref{eucl1} a direct computation yields
$$
(AC - CA)^\lambda_\rho \,\, B^\alpha_\beta D^\delta_\gamma + 
(CA - AC)^\lambda_\rho \,\, B^\delta_\gamma D^\alpha_\beta = 0
$$
which is then equivalent to the vanishing of the commutator $[A,C]=0$. 
The condition $[B, D] = 0$ follows in a similar fashion from the second condition in \eqref{eucl1}. 
\end{proof}
\end{prop}

\begin{prop}\label{solgen}
Let $A$ and $C$ be the two commuting real $N_1 \times N_1$ matrices with $A$
symmetric and $C$ antisymmetric and let $B$ and $D$ be two
commuting real $N_2 \times N_2$ matrices with $B$ symmetric and $D$
antisymmetric: 
\beq\label{consgen}
[A, C] = 0 = [B,D] ,
\eeq
\begin{align}\label{matgen}
{}^t A = A = \overline{A} , \qquad
{}^t B = B = \overline{B}  , \qquad
- \, {}^t C = C = \overline{C}  , \qquad
- \, {}^t D = D = \overline{D}  , 
\end{align}
Assume in addition that 
\begin{align}\label{centgen}
A^2 \otimes B^2 + C^2 \otimes D^2 = \II_{N_1} \otimes \II_{N_2}.
\end{align}
Then the matrix $R^{\lambda \alpha}_{\beta \mu}$ of the form 
\beq\label{rgen}
R^{\lambda \alpha}_{\beta \mu} 
= A^{\lambda}_ {\mu} \, B^{\alpha}_{\beta} + \ii \, C^{\lambda}_ {\mu} \, D^{\alpha}_{\beta} 
\eeq
satisfies conditions \eqref{eucl1} as well as conditions \eqref{eucl0}.
\end{prop}

\begin{proof}
From proposition \ref{pro-comm} the vanishing of the commutators \eqref{consgen} is equivalent to conditions \eqref{eucl1}. On the other hand, the vanishing of the commutator 
and all matrices being real, the `imaginary part' in \eqref{real8} vanishes and this condition reduces to 
\eqref{centgen}. Finally a direct computation also shows that the reality of the matrices and their 
(anti)symmetry implies that conditions \eqref{eucl0} are satisfied as well.
\end{proof}

\section{Quaternions and the four-dimensional euclidean space}\label{qr4}

We shall present explicit solutions for the matrix $R^{\lambda \alpha}_{\beta \mu}$ 
of the type in \eqref{rgen},
by using several results on the geometry of quaternions and related spaces that we present first.

\subsection{From quaternions to the euclidean space}
The space of quaternions $\IH$ is identified with $\IR^4$ in the usual way:
\beq\label{qrid}
\IH \ni q = x^0 e_0 + x^1 e_1 + x^2 e_2 + x^3 e_3 \quad \longmapsto \quad x = (x^\mu) = 
(x^0, x^1, x^2, x^3) \in \IR^4 .
\eeq
Here $e_0=1$ and the imaginary units $e_a$ obey the multiplication rule of the algebra $\IH$:
$$
e_a e_b = - \delta_{ab} + \sum_{c=1}^3\varepsilon_{abc} e_c .
$$
From this it follows an identification of the unit quaternions $\IH_1 = \{q \in \IH \, | \, q \bar q = 1 \} $ with the euclidean three-sphere $\IS^3 = \{x \in \IR^4 \, ; \, \|x\|^2 = \sum_\mu (x^\mu)^2 = 1\}$. 

With the identification \eqref{qrid}, left and right multiplication of quaternions are represented 
by matrices acting on $\IR^4$:  
$$
L_{q'} q := q'q \quad \to \quad E^+_{q'} (x) \qquad \mbox{and} \qquad 
R_{q'} q := qq' \quad \to \quad E^-_{q'} (x) .
$$
If $q$ is a unit quaternion, both $E^+_{q}$ and $E^-_{q}$ are orthogonal matrices. In fact
the unit quaternions form a subgroup of the multiplicative group $\IH^*$ of non vanishing quaternions.
When restricting to these, one has then:
\beq\label{su2}
\{ E^\pm_{q} \, | \, q \in \IH_1 \} \simeq \SU(2),  
\eeq
that is $E^+_{q}$ and $E^-_{q}$, for $q\in \IH_1$, are commuting $\SU(2)$ actions (each in the `defining representation')
on $\IR^4$, or taken together an action of $\SU(2)_L \times \SU(2)_R$ on $\IR^4$, with $L/R$ denoting left/right action.
This action is in fact the adjoint one, thus an action of $\SO(4) = \SU(2)_L \times \SU(2)_R / \IZ_2$. 

Let us denote $E^{\pm}_{a} = E^{\pm}_{e_a}$ for the imaginary units. 
By definition one has that 
\begin{align*}
E^{+}_{a} E^{-}_{b} = E^{-}_{b} E^{+}_{a} , \qquad  E^{\pm}_{a} E^{\pm}_{b} = - \delta_{ab} \II 
\pm \sum_{c=1}^3\varepsilon_{abc} E^{\pm}_{c} .
\end{align*}
In the following, it will turn out to be more convenient to change a sign to the `right' matrices: we shall rather use matrices 
$J^{+}_{a}:=E^{+}_{a}$ and $J^{-}_{a}:=-E^{-}_{a}$. For these one has
\begin{align*}
J^{+}_{a} J^{-}_{b} = J^{-}_{b} J^{+}_{a} , \qquad 
J^{\pm}_{a} J^{\pm}_{b} = - \delta_{ab} \II + \sum_{c=1}^3\varepsilon_{abc} J^{\pm}_{c} ,
\end{align*}
that is the matrices $J^\pm_a$ are two copies of the quaternionic imaginary units. 
With the identification \eqref{su2}, when acting on $\IR^4$, the matrices $J^\pm_a$ are a  representation of 
the Lie algebra $\su(2)$ of $\SU(2)$
or, taken together a representation of $\su(2)_L \oplus \su(2)_R$.  

Indeed, by the definition above one can explicitly compute the expressions
\beq
(J^{\pm}_{a})_{\mu\nu} = \mp ( \delta_{0\mu}  \delta_{a\nu} - \delta_{a\mu}  \delta_{0\nu})
+ \sum_{b,c=1}^3\varepsilon_{abc} \delta_{b\mu} \delta_{c\nu} , \quad \mbox{for} \quad a=1,2,3 ,
\eeq
for the components of the real $4 \times 4$  matrices $J^{\pm}_{a}$.

For the standard positive definite scalar product on $\IR^4$, the six matrices $J^{\pm}_{a}$ are readily checked to be antisymmetric, ${}^t J^{\pm}_{a} = - J^{\pm}_{a}$ and one finds in addition that 
$$
-\tfrac{1}{4} \tr (J^{\pm}_{a} J^{\pm}_{a}) =\delta_{ab} .
$$
On the other hand, the nine matrices $J^{+}_{a} J^{-}_{b}$ form an orthonormal basis for the space of symmetric trace-less matrices. In fact, the data $(\II, J^{+}_{a}, J^{-}_{a}, J^{+}_{a} J^{-}_{b})$ form an orthonormal basis of the endomorphisms algebra of $\IR^4$ (that is matrices) adapted to the decomposition 
\begin{align*}
\IR^4 \otimes \IR^4{}^* \simeq \IR^4{}^* \otimes \IR^4 & = D^{(0,0)} \oplus D^{(1,0)} \oplus D^{(0,1)} \oplus D^{(1,1)} \\
& = \IR \II  \oplus \Lambda^2_+ \IR^4{}^* \oplus   \Lambda^2_- \IR^4{}^* \oplus S^2_{(0)} \IR^4{}^* ,
\end{align*}
into irreducible $\SO(4)$-invariant subspaces. Here $S^2_{(0)} \IR^4{}^*$ is the space of trace-less elements of the degree-two part of the symmetric algebra over $\IR^4{}^*$, while the space $\Lambda^2 \IR^4{}^*$ of exterior two-forms (anti-symmetric two-vectors) on $\IR^4$ is split into a 
self-dual and an anti-self-dual part,
$\Lambda^2 \IR^4{}^* = \Lambda^2_+ \IR^4{}^* \oplus   \Lambda^2_- \IR^4{}^*$, where:
$$
\Lambda^2_\pm \IR^4{}^* = \{F \in \Lambda^2 \IR^4{}^* \, | \, \star F = \pm F \} \, \qquad  
(\star F)_{\mu \nu} = \tfrac{1}{2} \sum_{\rho, \sigma=1}^4 \varepsilon_{\mu \nu \rho \sigma} F_{\rho \sigma} .
$$ 
When equipping the space $\Lambda^2 \IR^4{}^*$ with the scalar product 
$(F | F') = \tfrac{1}{4} \sum_{\mu, \nu=1}^4 F_{\mu \nu} F'_{\mu \nu}$, the splitting 
$\Lambda^2 \IR^4{}^* = \Lambda^2_+ \IR^4{}^* \oplus \Lambda^2_- \IR^4{}^*$ is an orthogonal one into three-dimensional irreducible subspaces $\Lambda^2_\pm \IR^4{}^*$ and the representation of $\SO(4)$ on 
$\Lambda^2 \IR^4{}^* = \Lambda^2_+ \IR^4{}^* \oplus \Lambda^2_- \IR^4{}^*$ induces an homomorphism, 
\beq\label{act4}
\pi: \SO(4) \to \SO(3) \times \SO(3), 
\eeq
with kernel $\pm \II$ (that is each factor is in the adjoint action of $\SU(2)$).
More precisely, $(J^{\pm}_{1}, J^{\pm}_{2}, J^{\pm}_{3})$ is canonically an orthonormal basis of 
$\Lambda^2_\pm \IR^4{}^* \simeq \IR^3$ considered as an oriented three-dimensional euclidean space with the orientation of this basis; mapping $J^{\pm}_{a} \to J^{\mp}_{a}$ amounts to exchange the orientation. 
All the above will be quite useful in the following.

\section{Noncommutative quaternionic planes} 
We are ready to exhibit solutions in dimension eight for the matrices in the Proposition 
\ref{solgen} by using the results presented in \S\ref{qr4}. In particular, we recall again that the matrices $(J^{\pm}_{1}, J^{\pm}_{2}, J^{\pm}_{3})$ are canonically an orthonormal basis of $\IR^3 \simeq \Lambda^2_\pm \IR^4{}^*$. Based on this, with a vector 
$\underbar{u} = (u^1, u^2, u^3) \in \IR^3$ we get matrices 
$$
J^{+}_{\underbar{u}} := u^1 J^{+}_{1} + u^2 J^{+}_{2} + u^3 J^{+}_{3} \qquad \mbox{or} \qquad 
J^{-}_{\underbar{u}} := u^1 J^{-}_{1} + u^2 J^{-}_{2} + u^3 J^{-}_{3} .
$$

\subsection{Toric families} 
For matrices $A,B,C,D$ we take
$$
A = u \, \II, \quad B = \II , \qquad C = J^{\pm}_{\underbar{u}} , \quad D = J^{\mp}_{\underbar{v}}, 
$$
for $u\in \IR$ and $\underbar{u}, \underbar{v} \in \IR^3$. One easy sees that conditions \eqref{consgen} and \eqref{matgen} are satisfied and these matrices lead to the two families::
$$
(R^\pm)^{\lambda \alpha}_{\beta \mu} = u \, \delta^\lambda_\mu \delta^\alpha_\beta +
\ii \, (J^{\pm}_{\underbar{u}})^\lambda_\mu \, (J^{\mp}_{\underbar{v}})^\alpha_\beta, 
$$
Before we look at condition \eqref{centgen} we recall the action \eqref{act4} of $\SO(4)$ on 
$\Lambda^2 \IR^4{}^* \simeq \IR^3 \oplus \IR^3$. Using this action, one can always rotate $\underbar{u}, \underbar{v}$ to a common direction, $\widehat{\underbar{n}}$ say. The resulting matrices 
$(R^\pm)^{\lambda \mu}_{\nu \rho}$ are then written as
\beq\label{r8dtheta}
(R^\pm)^{\lambda \alpha}_{\beta \mu}  = u \, \delta^\lambda_\mu \delta^\alpha_\beta +
\ii \, v \, (J^{\pm}_{\widehat{\underbar{n}}})^\lambda_\mu \, (J^{\mp}_{\widehat{\underbar{n}}})^\alpha_\beta, 
\eeq
with real parameters $u,v \in \IR$ which, from condition \eqref{centgen}, are constrained by 
$$
u^2 + v^2 = 1. 
$$
Thus the matrices \eqref{r8dtheta} are direct generalisation of the four-dimensional one in \eqref{r4d} and the 
resulting associated noncommutative manifolds are just the toric manifolds of \cite{cl01} and \cite{cdv02}. 
As already observed in \S\ref{se:4d}, mapping $v \to -v$ would exchange $R^+$ with $R^-$ and results into isomorphic algebras; the space of parameters is indeed $\IS^1 / \IZ^2$. 

\subsection{Quaternionic families}\label{qdf}
More generally, for matrices $A,B,C,D$ we take
$$
A = u^0 \, \II, \quad B = \II , \qquad C = J^{\pm}_{\underbar{v}} , \quad D = J^{\pm}_{\underbar{u}}, 
$$
for $u^0\in \IR$ and $\underbar{u}, \underbar{v} \in \IR^3$, which again satisfy conditions \eqref{consgen} and \eqref{matgen}. These yield:
\beq\label{r8dquat0}
(R^\pm)^{\lambda \alpha}_{\beta \mu} = u^0 \, \delta^\lambda_\mu \delta^\alpha_\beta +
\ii \, (J^{\pm}_{\underbar{v}})^\lambda_\mu \, (J^{\pm}_{\underbar{u}})^\alpha_\beta . 
\eeq
Using the action of $\SO(3)$ one can now rotate only $\underbar{v}$ to a fixed direction 
$\widehat{\underbar{n}}$,
and in this case the matrix \eqref{r8dquat0} has parameters $u^0\in\IR$ and $\underbar{u}\in\IR^3$ 
constrained by the \eqref{centgen} to
$$
(u^0)^2 + \underbar{u}^2 = 1,
$$
that is they make up a three-dimensional sphere $\IS^3$.  There is in fact a residual `gauge' freedom in that one can use a rotation around the direction $\widehat{\underbar{n}}$ to remove one component of the vector $\underbar{u}$. Thus if $\widehat{\underbar{n}}_1$ and $\widehat{\underbar{n}}_2$ are two orthogonal unit vectors, we get families of noncommutative spaces determined by the matrices 
$$
(R^\pm)^{\lambda \alpha}_{\beta \mu} = u^0 \, \delta^\lambda_\mu \delta^\alpha_\beta +
\ii \, (J^{\pm}_{\widehat{\underbar{n}}_1})^\lambda_\mu \, (J^{\pm}_{\underbar{u}})^\alpha_\beta , \qquad \mbox{with} \quad \underbar{u} = u^1 \widehat{\underbar{n}}_1 + u^2 \widehat{\underbar{n}}_2  
$$
and parameters constrained by 
$$
(u^0)^2 + (u^1)^2 + (u^2)^2 = 1 ,
$$
that is a two-dimensional sphere $\IS^2 = \IS^3 / \IS^1$. 
Thus we have natural generalisations of the toric 
four-dimensional noncommutative spaces described in \S\ref{se:4d} for which the space of deformation parameters is $\IS^1 / \IZ_2$.

\subsection{Additional families} 
There are two additional classes of families which are not connected to the classical case. One could start with general parameters and with suitable rotations like the ones used in the previous cases, transform them in a simpler form. 

\subsubsection{Stratum I} 
With $\widehat{\underbar{n}}$ a unit vector and $u\in\IR$, consider matrices 
$$
C = u \, J^{\pm}_{\widehat{\underbar{n}}} , \qquad 
D = J^{\mp}_{\widehat{\underbar{n}}} .
$$
Matrices $C$ and $D$ commuting with $A$ and $B$ respectively, can then be taken to be
$$
A = J^{\pm}_{\widehat{\underbar{n}}} J^{\mp}_{\underbar{v}} , \qquad 
B = J^{\mp}_{\widehat{\underbar{n}}} J^{\pm}_{\underbar{w}}
$$
for vectors $\underbar{v}, \underbar{w} \in \IR^3$. For the corresponding matrices 
$$
(R^\pm)^{\lambda \alpha}_{\beta \mu} = 
(J^{\pm}_{\widehat{\underbar{n}}} J^{\mp}_{\underbar{v}})^\lambda_\mu \, 
(J^{\mp}_{\widehat{\underbar{n}}} J^{\pm}_{\underbar{w}})^\alpha_\beta \, + \, \ii \, u \, (J^{\pm}_{\widehat{\underbar{n}}})^\lambda_\mu 
(J^{\mp}_{\widehat{\underbar{n}}})^\alpha_\beta \, , 
$$
the condition \eqref{centgen} yields the constraint:
$$
(\underbar{v})^2  (\underbar{w})^2 + u^2 = 1 .
$$

\subsubsection{Stratum II} 
With $\widehat{\underbar{n}}_1$ and $\widehat{\underbar{n}}_2$ two orthogonal unit vectors, consider matrices 
$$
C = J^{\pm}_{\widehat{\underbar{n}}_1} , \qquad 
D = J^{\pm}_{\underbar{u}} \qquad \mbox{with} 
\quad \underbar{u} = u^1 \widehat{\underbar{n}}_1 + u^2 \widehat{\underbar{n}}_2 .
$$
Matrices $C$ and $D$ commuting with $A$ and $B$ respectively, can then be taken to be
$$
A = J^{\pm}_{\widehat{\underbar{n}}_1} J^{\mp}_{\underbar{v}} , \qquad 
B = J^{\pm}_{\underbar{u}} J^{\mp}_{\underbar{w}}
$$
for vectors $\underbar{v}, \underbar{w} \in \IR^3$. 
For the corresponding matrices 
$$
(R^\pm)^{\lambda \alpha}_{\beta \mu} = 
(J^{\pm}_{\widehat{\underbar{n}}_1} J^{\mp}_{\underbar{v}})^\lambda_\mu \, 
(J^{\pm}_{\underbar{u}} J^{\mp}_{\underbar{w}})^\alpha_\beta \, + \, 
\ii t \, 
(J^{\pm}_{\widehat{\underbar{n}}_1} )^\lambda_\mu 
(J^{\pm}_{\underbar{u}})^\alpha_\beta \, , 
$$
the condition \eqref{centgen} yields the constraint:
$$
\big( (u^1)^2 + (u^2)^2 \big)\big( t^2 + (\underbar{w})^2 (\underbar{v})^2 \big) = 1 .
$$

\section{Tori and Spheres}
The deformed $*$-algebras we have defined in the previous section are coordinate algebras 
on `spaces' defined by duality, that we shall generically denote $\IR^8_{\bf u}$. 
We now concentrate on the noncommutative spaces that we have named quaternionic in \S\ref{qdf}. As it will emerge these come with a rich symmetry structure and are particularly well behaved.

The noncommutative plane $\IR^8_{\bf u}$ has coordinate algebra $\cau = \ca(\IR^8_{\bf u})$ generated by elements
$(x_1, x_2)=(x_1^\mu, x_2^\alpha)$, for $\mu,\alpha\in\{0,1, \dots, 3\}$ subject to relations:
\begin{align}\label{r8+-}
x_1^\lambda x_1^\mu & = x_1^\mu x_1^\lambda , \qquad x_2^\alpha x_2^\beta = x_2^\beta 
x_2^\alpha , \nn \\
x_1^\lambda x_2^\alpha & = R^{\lambda \alpha}_{\beta \rho} \, x_2^\beta x_1^\rho , \nn \\
& = u^0 \, x_2^\alpha x_1^\lambda + \ii \, (J^{\pm}_{\underbar{u}})^\alpha_\beta x_2^\beta 
\, (J^{\pm}_{\widehat{\underbar{n}}_1})^\lambda_\rho x_1^\rho \nn \\
& = u^0 \, x_2^\alpha x_1^\lambda + \ii \, J^{\pm}_{\underbar{u}}(x_2)^\alpha 
\, J^{\pm}_{\widehat{\underbar{n}}_1}(x_1)^\lambda ,
\end{align}
with parameter ${\bf u} = (u^0, u^1, u^2) \in \IS^2$. In our construction a prominent role is played by the quaternionic geometry as illustrated in \S\ref{qr4}. 
From the way the two noncommutative algebras $\cau$ have been constructed we may as well think of each of them as being the coordinate 
algebra of two noncommutative copies of $\IH$, that is of an isomorphism
$$
\IR^8_{\bf u} \simeq \IH \times_{\bf u} \IH . 
$$
In fact the two algebras $\cau$ are isomorphic with isomorphism 
$\capo \to \cane$ given by 
$$
(x_1^0, x_1^1, x_1^2, x_1^3) \mapsto (x_1^0, -x_1^1, -x_1^2, -x_1^3) , 
\quad  (x_2^0, x_2^1, x_2^2, x_2^3) \mapsto (x_2^0, -x_2^1, -x_2^2, -x_2^3). 
$$
This map exchanges the orientation in each of the two copies which, from the considerations at the end of \S\ref{qr4}, is the same as the exchange 
$J^{\pm}_{a} \to J^{\mp}_{a}$.

\subsection{Noncommutative quaternionic tori} 
As it is well known, there is an identification of the unit quaternions $\IH_1$
with the euclidean three-sphere $\IS^3$. Now, if $\IC_1$ denotes the unit complex number, 
the two-dimensional torus is written as $\IT^2 = \IC_1 \times \IC_1 = \IS^1 \times \IS^1$. 
Then, in analogy with this, and for lack of a better name, we shall call a \emph{quaternionic torus} the cartesian product 
of two copies of $\IH_1$, that is 
$$
\IT^{\IH} = \IH_1 \times \IH_1 \simeq \IS^3 \times \IS^3. 
$$ 

And in analogy with the passage from the torus $\IT^2$ to the noncommutative one $\IT^2_\theta$ which is behind the noncommutative four-dimensional planes described in \S\ref{qr4}, we can pass from the torus 
$\IT^{\IH}$ to a noncommutative quaternionic torus 
$$
\IT^{\IH}_{\bf u} = \IH_1 \times_{\bf u} \IH_1 = \IS^3 \times_{\bf u} \IS^3 \qquad \mbox{for} \quad {\bf u} \in \IS^2. 
$$ 
The space $\IT^{\IH}_{\bf u}$ is defined by duality as having coordinate algebra 
$$
\ca(\IT^{\IH}_{\bf u}) = \ca(\IR^8_{\bf u}) /< \|x_1\|^2 - \II, \|x_2\|^2 - \II> , 
$$
that is the quotient of the algebra given by the relations \eqref{r8+-}, modulo the ideal generated 
by the conditions $\|x_1\|^2 = 1 = \|x_2\|^2$. The quotient is well defined since the quadratic elements 
$\|x_1\|^2 = \sum_\mu (x_1^\mu)^2$ and $\|x_2\|^2 = \sum_\alpha (x_2^\alpha)^2$ are both central.  

\subsection{Noncommutative seven-spheres} 
The quadratic element $\|x\|^2 = \|x_1\|^2 + \|x_2\|^2$ is central as well. We get a noncommutative 
seven-spheres $\IS^{7}_{\bf u}$ by the quotient:
$$
\ca(\IS^{7}_{\bf u}) = \ca(\IR^8_{\bf u}) /< \|x\|^2 - \II > .
$$
Again we may think of these spheres as being $\IS^{7}_{\bf u} \simeq (\IH^2_1)_{\bf u}$.

\subsection{An $\SU(2) \times \SU(2)$ symmetry} 
We know from \S\ref{qr4} that for $q\in\IH_1$, a unit quaternion, the matrices $J^\pm_q$ 
are a faithful representation of $\SU(2)_R$ or $\SU(2)_L$, with $J^+_q$ giving the right action 
while $J^-_q$ the left one.  Since the matrices commute, $J^{+}_{q} J^{-}_{q'} = J^{-}_{q'} J^{+}_{q}$, 
the relations in \eqref{r8+-} for $\capo$ are invariant under the right $\SU(2)\times\SU(2)$ action given by
\beq
x_1^\lambda \to  (J^{-}_{q_1})^\lambda_\nu \, x_1^\nu , \qquad 
x_2^\alpha \to  (J^{-}_{q_2})^\alpha_\beta \, x_2^\beta , \qquad \mbox{for} \quad q_1, q_2 \in\IH_1 ,
\eeq
while relations in \eqref{r8+-} for $\cane$ are invariant under the left $\SU(2)\times\SU(2)$ action given by
\beq
x_1^\lambda \to  (J^{+}_{q_1})^\lambda_\nu \, x_1^\nu , \qquad 
x_2^\alpha \to  (J^{+}_{q_2})^\alpha_\beta \, x_2^\beta   , \qquad \mbox{for} \quad q_1, q_2 \in\IH_1 .  
\eeq

\appendix

\section{Quadratic algebras}\label{quad}

To be definite we take the ground field to be complex numbers $\IC$. 
A homogeneous \emph{quadratic algebra} \cite{ma88}, \cite{pp05} 
is an associative algebra $\ca$ of the form
$$
\ca=A(E,R)=T(E)/(R)
$$
with $E$ a finite-dimensional vector space, $R$ a subspace of $E\otimes E$ and $(R)$ denoting the two-sided ideal of the tensor algebra $T(E)$ over $E$ generated by $R$. The space $E$ is the space of generators of $\ca$ and the subspace $R$ of $E\otimes E$ is the space of relations of $\ca$. The algebra $\ca=A(E,R)$ is naturally a graded algebra $\ca=\bigoplus_{n\in \mathbb N}\ca_n$ which is connected, that is such that $\ca_0=\IC \II$ and generated by the degree 1 part, $\ca_1=E$.

To a quadratic algebra $\ca=A(E,R)$ as above one associates another quadratic algebra, \emph{its Koszul dual} $\ca^!$, defined by
$$
\ca^!=A(E^\ast,R^\perp)
$$
where $E^\ast$ denotes the dual vector space of $E$ and $R^\perp\subset E^\ast \otimes E^\ast$ is the orthogonal of the space of relations $R\subset E\otimes E$ defined by 
$$
R^\perp = \{\omega \in E^\ast\otimes E^\ast \, ; \, \langle\omega,r\rangle=0,\forall r\in R\} .
$$
As usual, by using the finite-dimensionality of $E$, one identifies $E^\ast \otimes E^\ast$ with the dual vector space $(E\otimes E)^\ast$ of $E\otimes E$. One has of course $(\ca^!)^!=\ca$
and the dual vector spaces $\ca^{!\ast}_n$ of the homogeneous components $\ca^!_n$ of $\ca^!$ are 
\beq
\ca^{!\ast}_1=E \qquad 
\mbox{and} \qquad 
\ca^{!\ast}_n=\bigcap_{r+s+2=n} E^{\otimes^r}\otimes R \otimes E^{\otimes^s}
\label{En}
\eeq
for $n\geq 2$, as easily verified. In particular $\ca^{!\ast}_2=R$ and $\ca^{!\ast}_n\subset E^{\otimes^n}$ for any $n\in \mathbb N$.

Consider the sequence of free left $\ca$-modules
\begin{equation}
K(\ca): \qquad\quad \cdots \stackrel{b}{\rightarrow}\ca\otimes \ca^{!\ast}_{n+1}\stackrel{b}{\rightarrow}\ca\otimes \ca^{!\ast}_n\rightarrow \dots \rightarrow \ca\otimes \ca^{!\ast}_2\stackrel{b}{\rightarrow}\ca\otimes E\stackrel{b}{\rightarrow}\ca\rightarrow 0
\label{K}
\end{equation}
where $b:\ca\otimes \ca^{!\ast}_{n+1}\rightarrow \ca\otimes \ca^{!\ast}_n$ is induced by the left $\ca$-module homomorphism of $\ca\otimes E^{\otimes^{n+1}}$ into $\ca\otimes E^{\otimes^n}$ defined by
$$
b(a\otimes (x_0\otimes x_1\otimes \dots \otimes x_n))=ax_0 \otimes (x_1\otimes \dots \otimes x_n)
$$
for $a\in \ca$, $x_k \in E$. It follows from \eqref{En} that $\ca^{!\ast}_n\subset R\otimes E^{\otimes^{n-2}}$ for $n\geq 2$, which implies that $b^2=0$. As a consequence, the sequence $K(\ca)$ in \eqref{K} is a chain complex of free left $\ca$-modules called the \emph{Koszul complex} of the quadratic algebra $\ca$. 
The quadratic algebra $\ca$ is said to be a \emph{Koszul algebra} whenever its Koszul complex is acyclic in positive degrees, that is, whenever $H_n(K(\ca))=0$ for $n\geq 1$. One shows easily that $\ca$ is a Koszul algebra if and only if its Koszul dual $\ca^!$ is a Koszul algebra.

It is important to realize that the presentation of $\ca$ by generators and relations is equivalent to the exactness of the sequence
$$
\ca\otimes R \stackrel{b}{\rightarrow} \ca\otimes E \stackrel{b}{\rightarrow}\ca \stackrel{\varepsilon}{\rightarrow} \IC\rightarrow 0
$$
with $\varepsilon$ the map induced by the projection onto degree $0$. Thus one always has
$$
H_1(K(\ca))=0 \qquad \mbox{and} \qquad H_0(K(\ca))=\IC
$$
and, whenever $\ca$ is Koszul, the sequence
$$
K(\ca)\stackrel{\varepsilon}{\rightarrow} \IC \rightarrow 0 ,
$$
is a free resolution of the trivial module $\IC$. This resolution is then a minimal projective resolution of $\IC$ in the category of graded modules \cite{ca58}. 

Let $\ca=A(E,R)$ be a quadratic Koszul algebra such that $\ca^!_D\not=0$ and $\ca^!_n=0$ for $n>D$. 
Then the trivial (left) module $\IC$ has projective dimension $D$ which implies that $\ca$ has global dimension $D$ (see \cite{ca58}). This also implies that the Hochschild dimension of $\ca$ is $D$ (see \cite{be05}). 
By applying the functor $\Hom_\ca(\, \cdot \, , \ca)$ to the Koszul chain complex $K(\ca)$ of left $\ca$-modules one obtains the cochain complex $L(\ca)$ of right $\ca$-modules 
$$
L(\ca) : \qquad \quad 0\rightarrow \ca\stackrel{b'}{\rightarrow} \dots\stackrel{b'}{\rightarrow}\ca^!_n \otimes \ca \stackrel{b'}{\rightarrow} \ca^!_{n+1} \otimes \ca \stackrel{b'}{\rightarrow} \cdots \, .
$$
Here $b'$ is the left multiplication by $\sum_k\theta^k\otimes e_k$ in $\ca^!\otimes \ca$ where $(e_k)$ is a basis of $E$ with dual basis $(\theta^k)$. The algebra $\ca$ is said to be \emph{Koszul-Gorenstein} if it is Koszul of finite global dimension $D$ as above and if $H^n(L(\ca))=\IC \, \delta^n_D$. Notice that this implies that $\ca^!_n\simeq \ca^{!\ast}_{D-n}$ as vector spaces (this is a version of Poincar\'e duality).

Finally, a graded algebra $\ca=\oplus_n \ca_n$ is said to have \emph{polynomial growth} whenever there are a positive $C$ and $N \in \IN$ such that, for any $n \in \IN$, 
$$
\dim (\ca_n)\leq C n^{N-1} .
$$

\medskip
As before, let $E$ be a finite-dimensional vector space with the tensor algebra $T(E)$ 
endowed with its natural filtration $F^n(T(E))=\bigoplus_{m\leq n} E^{\otimes^m}$.
A {\sl nonhomogeneous quadratic algebra} \cite{po93}, \cite{bg96}, 
is an algebra $\fraca$ of the form
$$
\fraca=A(E,P)=T(E)/(P)
$$
where $P$ is a subspace of $F^2(T(E))$ and where $(P)$ denotes as above the two-sided ideal of $T(E)$
generated by $P$. The filtration of the tensor algebra $T(E)$ induces a filtration $F^n(\fraca)$ of $\fraca$ and one associates to $\fraca$ the graded algebra
$$
\mbox{gr} (\fraca)=\oplus_n F^n(\fraca)/F^{n-1}(\fraca) .
$$
Let $R$ be the image of $P$ under the canonical projection of $F^2(T(E))$ onto $E\otimes E$ and let $\ca=A(E,R)$ be the homogeneous quadratic algebra $T(E)/(R)$; this $\ca$ is referred to as the {\sl quadratic part} of $\fraca$. There is a canonical surjective graded algebra homomorphism 
$$
\mbox{can} :\ca\rightarrow \mbox{gr}(\fraca) .
$$
One says that $\fraca$ has the {\sl  Poincar\'e-Birkhoff-Witt (PBW) property} whenever this homomorphism is an isomorphism. The terminology  comes from the example 
where $\fraca=U(\fracg)$ the universal enveloping algebra of a Lie algebra $\fracg$.
A central result (see  \cite{bg96} and \cite{pp05}) states that if $\fraca$ has the PBW property then the following conditions are satisfied: 
\begin{align}\label{i-ii}
& \mathrm{(i)} \quad P\cap F^1 (T(E))=0 , \nn \\
& \mathrm{(ii)} \quad (P \cdot E + E \cdot P)\cap F^2 (T(E)) \subset P .
\end{align}
\noindent 
Conversely, if the quadratic part $\ca$ is a Koszul algebra, the conditions 
$\mathrm{(i)}$ and $\mathrm{(ii)}$ imply that $\fraca$ has the PBW property.
 Condition (i) means that $P$ is obtained from $R$ by adding to each non-zero element of $R$ terms of degrees 1 and 0. That is there are linear mappings $\psi_1:R\rightarrow E$ and $\psi_0:R
 \rightarrow \IC$ such that one has
\beq
P=\{r+\psi_1(r)+\psi_0(r) \II \,\, \vert \,\, r \in R\}
\label{PR}
\eeq
giving $P$ in terms of $R$. Condition (ii) is a generalisation of the Jacobi identity (see \cite{pp05}).


\end{document}